\documentclass{article}

\usepackage[english]{babel}

\usepackage[a4paper,top=2cm,bottom=2cm,left=3cm,right=3cm,marginparwidth=1.75cm]{geometry}

\usepackage{comment}
\usepackage{amsmath, amssymb}
\usepackage{graphicx}
\usepackage{xcolor}
\definecolor{link-color}{RGB}{60,60,120} %Not so shiny blue links
\usepackage[colorlinks=true, allcolors=link-color]{hyperref}
\usepackage{mathtools}
\usepackage{amsfonts}
\usepackage{amsthm}
\usepackage{tikz-cd}
\usepackage{nicematrix}

\usepackage{thmtools}
\usepackage{thm-restate}
\declaretheorem[name=Theorem,numberwithin=section]{theorem}

\usepackage{cleveref}
\renewcommand{\nameref}[1]{\hyperref[#1]{\nameCref{#1}}}

\usepackage{listings}

\newcommand{\highlight}{\textbf}

\newtheorem{lemma}[theorem]{Lemma}

\newtheorem{definition}[theorem]{Definition}
\newtheorem{remark}[theorem]{Remark}
\newtheorem{example}[theorem]{Example}
\numberwithin{equation}{section}
\newcommand{\mycomment}[1]{}

\newcommand{\RR}{\mathbb R}

\newcommand{\Prr}{P_r}

\DeclareMathOperator{\im}{Im}

\title{Exact values of generic subrank}
\author{Paweł Pielasa, Matouš Šafránek, Anatoli Shatsila}

\begin{document}

\maketitle
\begin{abstract}
    In this article we prove the subrank of a generic tensor in $\mathbb{C}^{n,n,n}$ to be $Q(n) = \lfloor\sqrt{3n - 2}\rfloor$ by providing a lower bound to the known upper bound. More generally, we find the generic subrank of tensors of all orders and dimensions.
    This answers two open questions posed in \cite{10.4230/LIPIcs.CCC.2022.9}. Finally, we compute dimensions of varieties of tensors of subrank at least $r$.
\end{abstract}

\section{Introduction}

Geometry of tensors is a subject of great importance in mathematics, physics and computer science with numerous applications in complexity theory (fast matrix multiplication, P vs NP problem, etc.) \cite{JML, JML2}, quantum information theory \cite{JML3} and other areas of science ranging from signal processing to chemistry \cite{doi:10.1137/07070111X}. 

The subrank of a tensor is an important notion in algebraic complexity theory introduced by Strassen in \cite{Strassen87} to analyse the complexity of matrix multiplication. As a dual notion to tensor rank, it serves as one of the crucial invariants of tensors that has been extensively researched over the years. Recently, numerous other invariants of tensors, such as the slice rank \cite{tao}, the analytic rank \cite{Gowers}, and the geometric rank \cite{kop}, have been introduced and investigated. It is known that the subrank provides a lower bound for all invariants mentioned above.

In this work we are interested in generic values of the subrank, that is, the subrank of "almost all" tensors in a given tensor space. More precisely, one can show \cite[Proposition 2.1]{10.4230/LIPIcs.CCC.2022.9} that for any tensor space there is a Zariski open subset of tensors where subrank is constant. We call this number \textit{the generic subrank} and denote the subrank of a generic tensor $T \in K^{n_1,n_2,n_3}$ by $Q(n_1,n_2,n_3)$ with a shorthand $Q(n)=Q(n,n,n)$.

Strassen \cite{Strassen91} and  Bürgisser \cite{Burg} found the upper bound for the generic subrank $Q(n) \leq n^{2/3 + o(1)}$ but exact asymptotic values have been unknown for a long time. Only recently it has been shown in \cite{10.4230/LIPIcs.CCC.2022.9} that, in fact, $Q(n) = \Theta(\sqrt{n})$. Interestingly, the generic value of tensor parameters introduced above is the maximal value $n$, hence the subrank is much lower than aforementioned ranks for a generic tensor. 

We note that determining precise generic values for tensor invariants is a difficult problem. For instance, generic values for a rank are unknown in general (see \cite{Strassen83, FRIEDLAND2012478} and Subsection \ref{genrank}).
We highlight one example where generic values are known. Namely, Alexander and Hirschowitz (see \cite{AH} and \cite[Chapter 5.4 and Chapter 15]{JML2}) showed that a generic symmetric $d$-tensor in $K^{n,\ldots,n}$ has generic symmetric rank $\left \lceil{\frac{\binom{n+d-1}{n}}{n}}\right \rceil$ apart from a known list of exception where the values are also known.

In this paper we find precise values of the generic subrank of tensors of any order and shape, thus providing a complete answer to the problem of determining values of the generic subrank.

\begin{comment}

- possible references:
Drisma
Vassilevska-williams
\end{comment}

\subsection{Results}

It was shown in \cite[Theorem 1.2 and Theorem 1.3]{10.4230/LIPIcs.CCC.2022.9} that $$3(\lfloor \sqrt{n/3+1/4} - 1/2 \rfloor)  \leq Q(n) \leq \lfloor \sqrt{3n-2} \rfloor.$$ They conjectured (see \cite[Section 6]{10.4230/LIPIcs.CCC.2022.9})  that the upper bound on $Q(n)$ is exactly tight and checked that it is indeed the case for $n \leq 100$ on a computer. In this note we show using a combinatorial method that the upper bound is indeed tight. 

\begin{restatable*}{corollary}{SimpleMainCor}
\label{simplemaincor}
    For any $n \in \mathbb{N}$ the generic subrank $Q(n)$ is equal to $\lfloor \sqrt{3n-2} \rfloor$ over any infinite field.
\end{restatable*}

Our proof works in a more general setting. We can thus confirm that other upper bounds given in \cite{10.4230/LIPIcs.CCC.2022.9} are exactly tight. The following general \nameref{maincor} is proved in Section \ref{main-section}.

\begin{restatable*}{theorem}{MainCor}
\label{maincor}
    For any $n_1,n_2,n_3 \in \mathbb{N}$ we have the generic subrank
    $$Q(n_1,n_2,n_3)=\min\left\{n_1,n_2,n_3,\left\lfloor \sqrt{n_1+n_2+n_3-2} \right\rfloor\right\}$$
    over any infinite field.
    %In particular, $Q(n)=\lfloor \sqrt{3n-2} \rfloor$. 
\end{restatable*}

We build on the result (see \autoref{thmgen}) from \cite{10.4230/LIPIcs.CCC.2022.9} which says that the generic subrank $Q(n)$ is at least $r$ if certain subspaces of $K^{r,r,r}$ generate the whole space. In \autoref{transfer} 
we straightforwardly reformulate this condition to a statement about linear independence of rows of a projection matrix $M$ whose non-zero entries are generic elements of $K$ (with possible repetitions). Finally, using a combinatorial argument we show that $M$ indeed has the linearly independent rows in \autoref{supertheorem}.

The proof can be carried out for higher order tensors analogously. This is carried out in Section \ref{higher-order}, where we get the following generalisation.

\begin{restatable*}{theorem}{HighOrderMainCor}
    Let $k\geq 3$. For any $n_1,\hdots,n_k \in \mathbb{N}$ we have 
    $$Q(n_1,\ldots,n_k)=\min\left\{n_1,\ldots,n_k,\left\lfloor (n_1+\hdots+n_k-(k-1))^{\frac{1}{k-1}} \right\rfloor\right\}$$
    over any infinite field.
    In particular, $Q(\underbrace{n,
    \ldots,n}_{k})=\lfloor (kn-k+1)^{\frac{1}{k-1}} \rfloor$. 
\end{restatable*} 

Thus, the bound from \cite[Theorem 1.5]{10.4230/LIPIcs.CCC.2022.9} is exactly tight .

Further, we adapt our combinatorial method to produce results not only about the locus of generic-subrank tensors, but also about the locus of higher subrank tensors. We confirm that sets of tensors in $K^{n_1,\ldots, n_k}$ of subrank at least $r$, denoted by $\mathcal{C}_r^{(n_1,\ldots,n_k)}$, have expected dimension.

\begin{restatable*}{theorem}{DimensionThm}
\label{dimension}
    Let $k\geq3$ and the underlying field be of characteristic $0$. For any $n_1,\dots n_k,r\in\mathbb{N}$ such that $Q(n_1,\dots,n_k)< r \leq n_1,\dots,n_k$ the following holds: $$\dim (\mathcal{C}_r^{(n_1,\ldots,n_k)}) = \prod_{i=1}^k n_i - r\left(r^{k-1} - \sum_{i=1}^k n_i + (k-1)\right).$$
\end{restatable*}

The conditions on $r$ and $n_i$ only restrict them to the non-trivial range, i.e. where $\mathcal{C}_r^{(n_1,\ldots,n_k)}$ is non-empty and not full-dimensional. We are thus able to calculate the dimensions of all $\mathcal{C}_r^{(n_1,\ldots,n_k)}$.

\section*{Acknowledgements}

We would like thank Mateusz Michałek for his continuous support and guidance during our stay at the University of Konstanz where this article was written. During the preparation of the article, Paweł Pielasa and Matouš Šafránek were supported by the DAAD/Ostpartnerschaften programme and Anatoli Shatsila was supported by the Erasmus+ programme.

\section{Main result}
\label{main-section}

\subsection{Set-up}

We denote by $[n]$ the set $\{1,2,\dots,n\}$ and by $\Prr$ the set of triples $(i,j,k) \in [r]^3$ with $i,j,k$ all different. We then recall the notation introduced in \cite{10.4230/LIPIcs.CCC.2022.9}.
For a field $K$ we denote $$K^{n_1, n_2, \ldots, n_k} := K^{n_1} \otimes K^{n_2} \otimes \hdots \otimes K^{n_k}.$$
For a tensor $T \in K^{n_1, n_2, \ldots, n_k}$ we define its \textit{subrank} as the maximal non-negative integer $s \leq n$ such that $I_s \leq T$, where $$I_s := I_1^{\oplus s} = \sum_{j=1}^s \bigotimes_{i=1}^{k}e_j^{(i)}$$ and $\{e_{j}^{(i)}\}_{j=1}^{n_i}$ is a basis of the corresponding $K^{n_i}$ for each $1 \leq i \leq k$. Note that this definition is independent of the choice of the bases.
We denote by $$\mathcal{C}_r^{(n_1,\ldots,n_k)} = \{T \in K^{n_1,\ldots,n_k} \: | \: Q(T) \geq r\}$$ the set of tensors of subrank at least $r$ in $K^{n_1,\ldots,n_k}$ and by $Q(n_1,\dots n_k)$ the subrank of a generic tensor in $K^{n_1,\dots,n_k}$ (more simply called generic subrank). For proof of existence see \cite[Proposition 2.1]{10.4230/LIPIcs.CCC.2022.9}.

\begin{remark}
    Generic subrank is the highest $r$ such that $\mathcal{C}_r^{(n_1,\ldots,n_k)}$ is Zariski-dense. Note that if $K = \mathbb{R}$, Zariski-dense $\mathcal{C}_r^{(n_1,\ldots,n_k)}$ may not be dense in the Euclidean topology. Indeed, one can show that, for example, $\mathcal{C}_2^{(2,2,2)} \subset \RR^{2,2,2}$ is not dense in the Euclidean sense, even though 2 is the generic subrank (\autoref{simplemaincor}). See \cite{Kruskai} and
    \cite[Section 7]{FRIEDLAND2012478} for details. 
\end{remark}

We define a linear subspace of $K^{r,r,r}$ as follows: $$W_r := \{T \in K^{r,r,r} \: | \: T_{ijk} = 0 \text{ if } (i,j,k) \in \Prr \} \subseteq K^{r,r,r}.$$
Equivalently,
$$W_r = \left\langle e_i\otimes e_j \otimes e_k\right\rangle_{(i,j,k) \in [r]^3\setminus \Prr}.$$
For a linear subspace $\mathcal{X}$ of $K^{n_2} \otimes K^{n_3}$ we define a linear subspace $$\mathcal{X}[1] := K^{n_1} \otimes \mathcal{X} \subseteq K^{n_1, n_2, n_3}.$$ This linear space consists precisely of the tensors whose slices in the first direction belong to $\mathcal{X}$.
We define $\mathcal{X}[2]$ and $\mathcal{X}[3]$ analogously.

Our proof relies on the following result.

\begin{theorem}[Theorem 3.9 \cite{10.4230/LIPIcs.CCC.2022.9}]
\label{thmgen}
    Let $r \leq n_1,n_2,n_3$ such that $n_i - r \leq r^2$ for $i = 1,2,3$. Let $\mathcal{X}_i \subseteq K^{r,r}$ be a generic subspace of dimension $n_i-r$ for $i=1,2,3$. Suppose $$\mathcal{X}_1[1] + \mathcal{X}_2[2] + \mathcal{X}_3[3] + W_r = K^{r,r,r},$$
    then $Q(n_1,n_2,n_3) \geq r$. Further, if characteristic of $K$ is 0, then we have the converse.
    
\end{theorem}

%We write out the subspaces in coordinates and translate this condition to a matrix. As a linear map, it just maps $\mathcal{X}_1[1]+ \mathcal{X}_2[2]+ \mathcal{X}_3[3]$ spanned by some generators to $K^{r,r,r}/W_r$ in the standard basis.
%To prove the span this matrix will need to have full row rank.

Working modulo $W_r$ we can rephrase the condition above as a statement about linear independence of rows of a projection matrix. 

\begin{definition}
\label{matrix}
Assume we are given $r, n_1,n_2,n_3\in\mathbb{N}$ such that $r\leq n_1,n_2,n_3$. We define a matrix $M \in K^{r(n_1+n_2+n_3 - 3r) \times r(r-1)(r-2)}$, whose columns are labelled by triples $(t,m,s)$ with $t \in [3],  m \in [r], s \in [n_t-r]$, and rows by triples $(i,j,k) \in \Prr$. The matrix $M$ is defined entry-wise by
    $$M^{t,m,s}_{i,j,k} = \begin{cases}
        a^{1,s}_{j,k},& t=1, m=i; \\
        a^{2,s}_{i,k},& t=2, m=j; \\
        a^{3,s}_{i,j},& t=3, m=k; \\
        0 & \text{otherwise},
    \end{cases}$$
    with variable entries $a^{t,s}_{i_1,i_2}$.%, where $t\in[3], s\in[n_t-r], i_1,i_2 \in[r]$. 
\end{definition}

Here the assumption $r\leq n_1,n_2,n_3$ is only to have the set $[n_t-r]$ defined.

\begin{lemma}
\label{transfer}
    With the assumptions of \autoref{thmgen} we have $$\mathcal{X}_1[1] + \mathcal{X}_2[2] + \mathcal{X}_3[3] + W_r = K^{r,r,r}$$ if and only if the matrix $M$ has generic rank (i.e. for generic values of $a^{t,s}_{i_1,i_2} \in K$) equal to $r(r-1)(r-2)$, that is full row rank.
\end{lemma}

\begin{proof}

    Denote the basis of $K^r$ by $e_1, \ldots, e_r$. Since $\mathcal{X}_t$ is generic for $t=1,2,3$ we can write 
    $$\mathcal{X}_1 = \left\langle \sum_{1 \leq j,k \leq r} a_{j,k}^{1,s} \; e_j \otimes e_k \right\rangle_{1 \leq s \leq n_1-r}$$
    $$\mathcal{X}_2 = \left\langle \sum_{1 \leq i,k \leq r} a_{i,k}^{2,s} \; e_i \otimes e_k \right\rangle_{1 \leq s \leq n_2-r}$$
    $$\mathcal{X}_3 = \left\langle \sum_{1 \leq i,j \leq r} a_{i,j}^{3,s} \; e_i \otimes e_j \right\rangle_{1 \leq s \leq n_3-r}$$
    for generic $a_{i,j}^{t,s} \in K$. 
    %This is just choosing a basis of matrices $a^{t,s}$ for each $\mathcal{X}_t$.
    It follows from the definition of $W_r$ that the basis of the space $K^{r,r,r}/W_r$ is given by $\{e_i \otimes e_j \otimes e_k\}_{i,j,k \in \Prr}$. On the other hand, the basis of $\mathcal{X}_1[1] \oplus \mathcal{X}_2[2] \oplus \mathcal{X}_3[3]$
    %todo maybe write \times would be more intuitive
    is given by vectors $v_{t,m,s}$ with $t \in [3],  m \in [r], s \in [n_t-r]$ and $$v_{1,m,s} = \sum_{1 \leq j,k \leq r}a_{j,k}^{1,s} \; e_m \otimes e_j \otimes e_k$$ and analogously for $t=2,3$. Note that the matrix $M$ defined in the statement is the matrix of the natural projection $$\pi: \mathcal{X}_1[1] \oplus \mathcal{X}_2[2] \oplus \mathcal{X}_3[3] \to K^{r,r,r}/W_r.$$ 
    Therefore, $$\mathcal{X}_1[1] + \mathcal{X}_2[2] + \mathcal{X}_3[3] + W_r = K^{r,r,r}$$ holds if and only if $\pi$ is surjective, or, equivalently, $M$ has linearly independent rows.

\end{proof}

\begin{figure}
    \centering
    \includegraphics[scale=1.2]{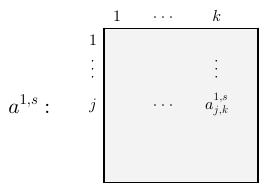} \\
    \includegraphics[scale=1.2]{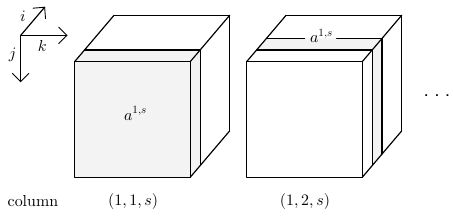}
\end{figure}

Now if we prove that the matrix $M$ has generic rank $r(r-1)(r-2)$ for any $r$ (keep in mind that $M$ depends on $r$), we get by the \autoref{transfer} and \autoref{thmgen} a lower bound on $Q(n_1,n_2,n_3)$.
In fact, the upcoming combinatorial argument proves that the matrix has generic rank equal to $r(r-1)(r-2)$ whenever it has at least as many columns as rows, that is whenever it is actually possible (\autoref{supertheorem}). That will immediately give us the highest possible lower bound, the generic subrank itself (\autoref{equivalence-remark}).

To give an idea about what we are working with, we provide an example.

\begin{example}
\label{matrix4}
Let $n_1=n_2=n_3=6$ and $r=4$. Below we present the matrix $M$. Dots indicate zeros and commas are omitted to save space. Indices of rows $(i,j,k)$ and columns $(t,m,s)$ are shown.

\begingroup\small$$\hskip-2.5cm \begin{pNiceMatrix}[first-row, first-col, code-for-first-row = \rotate]
 & 1,1,1 & 1,1,2 & 1,2,1 & 1,2,2 & 1,3,1 & 1,3,2 & 1,4,1 & 1,4,2 & 2,1,1 & 2,1,2 & 2,2,1 & 2,2,2 & 2,3,1 & 2,3,2 & 2,4,1 & 2,4,2 & 3,1,1 & 3,1,2 & 3,2,1 & 3,2,2 & 3,3,1 & 3,3,2 & 3,4,1 & 3,4,2 \\
123 & a^{11}_{23} & a^{12}_{23} & \cdot & \cdot & \cdot & \cdot & \cdot & \cdot & \cdot & \cdot & a^{21}_{13} & a^{22}_{13} & \cdot & \cdot & \cdot & \cdot & \cdot & \cdot & \cdot & \cdot & a^{31}_{12} & a^{32}_{12} & \cdot & \cdot \\
124 & a^{11}_{24} & a^{12}_{24} & \cdot & \cdot & \cdot & \cdot & \cdot & \cdot & \cdot & \cdot & a^{21}_{14} & a^{22}_{14} & \cdot & \cdot & \cdot & \cdot & \cdot & \cdot & \cdot & \cdot & \cdot & \cdot & a^{31}_{12} & a^{32}_{12} \\
132 & a^{11}_{32} & a^{12}_{32} & \cdot & \cdot & \cdot & \cdot & \cdot & \cdot & \cdot & \cdot & \cdot & \cdot & a^{21}_{12} & a^{22}_{12} & \cdot & \cdot & \cdot & \cdot & a^{31}_{13} & a^{32}_{13} & \cdot & \cdot & \cdot & \cdot \\
134 & a^{11}_{34} & a^{12}_{34} & \cdot & \cdot & \cdot & \cdot & \cdot & \cdot & \cdot & \cdot & \cdot & \cdot & a^{21}_{14} & a^{22}_{14} & \cdot & \cdot & \cdot & \cdot & \cdot & \cdot & \cdot & \cdot & a^{31}_{13} & a^{32}_{13} \\
142 & a^{11}_{42} & a^{12}_{42} & \cdot & \cdot & \cdot & \cdot & \cdot & \cdot & \cdot & \cdot & \cdot & \cdot & \cdot & \cdot & a^{21}_{12} & a^{22}_{12} & \cdot & \cdot & a^{31}_{14} & a^{32}_{14} & \cdot & \cdot & \cdot & \cdot \\
143 & a^{11}_{43} & a^{12}_{43} & \cdot & \cdot & \cdot & \cdot & \cdot & \cdot & \cdot & \cdot & \cdot & \cdot & \cdot & \cdot & a^{21}_{13} & a^{22}_{13} & \cdot & \cdot & \cdot & \cdot & a^{31}_{14} & a^{32}_{14} & \cdot & \cdot \\
213 & \cdot & \cdot & a^{11}_{13} & a^{12}_{13} & \cdot & \cdot & \cdot & \cdot & a^{21}_{23} & a^{22}_{23} & \cdot & \cdot & \cdot & \cdot & \cdot & \cdot & \cdot & \cdot & \cdot & \cdot & a^{31}_{21} & a^{32}_{21} & \cdot & \cdot \\
214 & \cdot & \cdot & a^{11}_{14} & a^{12}_{14} & \cdot & \cdot & \cdot & \cdot & a^{21}_{24} & a^{22}_{24} & \cdot & \cdot & \cdot & \cdot & \cdot & \cdot & \cdot & \cdot & \cdot & \cdot & \cdot & \cdot & a^{31}_{21} & a^{32}_{21} \\
231 & \cdot & \cdot & a^{11}_{31} & a^{12}_{31} & \cdot & \cdot & \cdot & \cdot & \cdot & \cdot & \cdot & \cdot & a^{21}_{21} & a^{22}_{21} & \cdot & \cdot & a^{31}_{23} & a^{32}_{23} & \cdot & \cdot & \cdot & \cdot & \cdot & \cdot \\
234 & \cdot & \cdot & a^{11}_{34} & a^{12}_{34} & \cdot & \cdot & \cdot & \cdot & \cdot & \cdot & \cdot & \cdot & a^{21}_{24} & a^{22}_{24} & \cdot & \cdot & \cdot & \cdot & \cdot & \cdot & \cdot & \cdot & a^{31}_{23} & a^{32}_{23} \\
241 & \cdot & \cdot & a^{11}_{41} & a^{12}_{41} & \cdot & \cdot & \cdot & \cdot & \cdot & \cdot & \cdot & \cdot & \cdot & \cdot & a^{21}_{21} & a^{22}_{21} & a^{31}_{24} & a^{32}_{24} & \cdot & \cdot & \cdot & \cdot & \cdot & \cdot \\
243 & \cdot & \cdot & a^{11}_{43} & a^{12}_{43} & \cdot & \cdot & \cdot & \cdot & \cdot & \cdot & \cdot & \cdot & \cdot & \cdot & a^{21}_{23} & a^{22}_{23} & \cdot & \cdot & \cdot & \cdot & a^{31}_{24} & a^{32}_{24} & \cdot & \cdot \\
312 & \cdot & \cdot & \cdot & \cdot & a^{11}_{12} & a^{12}_{12} & \cdot & \cdot & a^{21}_{32} & a^{22}_{32} & \cdot & \cdot & \cdot & \cdot & \cdot & \cdot & \cdot & \cdot & a^{31}_{31} & a^{32}_{31} & \cdot & \cdot & \cdot & \cdot \\
314 & \cdot & \cdot & \cdot & \cdot & a^{11}_{14} & a^{12}_{14} & \cdot & \cdot & a^{21}_{34} & a^{22}_{34} & \cdot & \cdot & \cdot & \cdot & \cdot & \cdot & \cdot & \cdot & \cdot & \cdot & \cdot & \cdot & a^{31}_{31} & a^{32}_{31} \\
321 & \cdot & \cdot & \cdot & \cdot & a^{11}_{21} & a^{12}_{21} & \cdot & \cdot & \cdot & \cdot & a^{21}_{31} & a^{22}_{31} & \cdot & \cdot & \cdot & \cdot & a^{31}_{32} & a^{32}_{32} & \cdot & \cdot & \cdot & \cdot & \cdot & \cdot \\
324 & \cdot & \cdot & \cdot & \cdot & a^{11}_{24} & a^{12}_{24} & \cdot & \cdot & \cdot & \cdot & a^{21}_{34} & a^{22}_{34} & \cdot & \cdot & \cdot & \cdot & \cdot & \cdot & \cdot & \cdot & \cdot & \cdot & a^{31}_{32} & a^{32}_{32} \\
341 & \cdot & \cdot & \cdot & \cdot & a^{11}_{41} & a^{12}_{41} & \cdot & \cdot & \cdot & \cdot & \cdot & \cdot & \cdot & \cdot & a^{21}_{31} & a^{22}_{31} & a^{31}_{34} & a^{32}_{34} & \cdot & \cdot & \cdot & \cdot & \cdot & \cdot \\
342 & \cdot & \cdot & \cdot & \cdot & a^{11}_{42} & a^{12}_{42} & \cdot & \cdot & \cdot & \cdot & \cdot & \cdot & \cdot & \cdot & a^{21}_{32} & a^{22}_{32} & \cdot & \cdot & a^{31}_{34} & a^{32}_{34} & \cdot & \cdot & \cdot & \cdot \\
412 & \cdot & \cdot & \cdot & \cdot & \cdot & \cdot & a^{11}_{12} & a^{12}_{12} & a^{21}_{42} & a^{22}_{42} & \cdot & \cdot & \cdot & \cdot & \cdot & \cdot & \cdot & \cdot & a^{31}_{41} & a^{32}_{41} & \cdot & \cdot & \cdot & \cdot \\
413 & \cdot & \cdot & \cdot & \cdot & \cdot & \cdot & a^{11}_{13} & a^{12}_{13} & a^{21}_{43} & a^{22}_{43} & \cdot & \cdot & \cdot & \cdot & \cdot & \cdot & \cdot & \cdot & \cdot & \cdot & a^{31}_{41} & a^{32}_{41} & \cdot & \cdot \\
421 & \cdot & \cdot & \cdot & \cdot & \cdot & \cdot & a^{11}_{21} & a^{12}_{21} & \cdot & \cdot & a^{21}_{41} & a^{22}_{41} & \cdot & \cdot & \cdot & \cdot & a^{31}_{42} & a^{32}_{42} & \cdot & \cdot & \cdot & \cdot & \cdot & \cdot \\
423 & \cdot & \cdot & \cdot & \cdot & \cdot & \cdot & a^{11}_{23} & a^{12}_{23} & \cdot & \cdot & a^{21}_{43} & a^{22}_{43} & \cdot & \cdot & \cdot & \cdot & \cdot & \cdot & \cdot & \cdot & a^{31}_{42} & a^{32}_{42} & \cdot & \cdot \\
431 & \cdot & \cdot & \cdot & \cdot & \cdot & \cdot & a^{11}_{31} & a^{12}_{31} & \cdot & \cdot & \cdot & \cdot & a^{21}_{41} & a^{22}_{41} & \cdot & \cdot & a^{31}_{43} & a^{32}_{43} & \cdot & \cdot & \cdot & \cdot & \cdot & \cdot \\
432 & \cdot & \cdot & \cdot & \cdot & \cdot & \cdot & a^{11}_{32} & a^{12}_{32} & \cdot & \cdot & \cdot & \cdot & a^{21}_{42} & a^{22}_{42} & \cdot & \cdot & \cdot & \cdot & a^{31}_{43} & a^{32}_{43} & \cdot & \cdot & \cdot & \cdot
\end{pNiceMatrix}$$
\endgroup

Note that each variable appears exactly twice in the matrix. Take for example $t=3$, then $a^{3,s}_{i,j}$ appears in row $(i,j,k)$ and column $(3,k,s)$ for all $k \in [4], k\not=i,j$. Hence, in general, it appears $r-2$ times.
\end{example}

\begin{example}
\label{cross-out}
Consider the matrix 
$$\begin{pmatrix}
a_1 & \cdot & a_3 & \cdot \\
\cdot & a_2 & a_1 & \cdot\\
\cdot & a_3 & \cdot & a_4 \\
a_4 & \cdot & \cdot & a_2
\end{pmatrix},$$
where dots indicate zeros.
%todo this is repeated
If we write its determinant as the sum of (up to) 24 terms, the monomial $a_1^2a_2a_3$ appears exactly once, as we will verify. Since the variable $a_1$ appears only twice in the matrix, to get $a_1^2$ as a factor in a determinant term, these two entries must be multiplied. However, since each factor in a determinant term must come from a different row and column, the remaining factors $a_2$, $a_3$ must come from the submatrix
$$\begin{pNiceMatrix}[create-medium-nodes, left-margin, right-margin]
    a_1 & \cdot & a_3 & \cdot \\
    \cdot & a_2 & a_1 & \cdot\\
    \cdot & a_3 & \cdot & a_4 \\
    a_4 & \cdot & \cdot & a_2
\CodeAfter
\begin{tikzpicture}[line width=0.3mm, rounded corners=1mm]
\node (1-1r) at (1-1) [draw,minimum width=0.5cm,minimum height=0.4cm] {};
\node (2-3r) at (2-3) [draw,minimum width=0.5cm,minimum height=0.4cm] {};
\draw (1-1r.east) -- (1-4-medium.east) ;
\draw (1-1r.south) -- (4-1-medium.south) ;
\draw (2-1-medium.west) -- (2-3r.west) ;
\draw (2-3r.east) -- (2-4-medium.east) ;
\draw (1-3-medium.north) -- (2-3r.north) ;
\draw (2-3r.south) -- (4-3-medium.south) ;

\end{tikzpicture}
\end{pNiceMatrix}
= \begin{pmatrix}
    a_3 & a_4 \\
    \cdot & a_2
\end{pmatrix}.$$

The computation above is straighforward, however, formally, one sees $a_2$ appears only once and then $a_3$ appears only once in the remaining $1\times 1$ submatrix, so there is only one way to multiply the matrix entries to get this monomial.
\end{example}

We aim to find such a unique monomial in our matrix $M$, because then it has full rank (\autoref{unique-enough}). Note that, however, our matrix $M$ is not necessarily a square matrix. In the following definition we formalise the procedure from the example above. 

\begin{definition}
\label{unique}
Let $M_0$ be a $k\times \ell$ matrix where each entry is either a variable from $\{a_1, a_2, \dots, a_n\}$ or zero, and no variable appears twice in one row or column. Define $M_i$ inductively as the submatrix of $M_{i-1}$ with all rows and columns containing $a_i$ crossed out. We call a monomial $m=a_1^{p_1}a_2^{p_2}\dots a_n^{p_n}$ of degree $k$ a \highlight{unique row monomial} of $M_0$ if $M_{i-1}$ contains $a_i$ exactly $p_i$ times for every $i$.
\end{definition}

%todo maybe put it in a definition or something
Later, when we decribe the algorithm to find such a monomial, we will not use this formal submatrix notation. Rather, as in \autoref{cross-out}, in each step we will pick a variable, find all its uncrossed occurrences (i.e. those in the submatrix) and cross them out with their rows and columns. What remains is the next submatrix. We will refer to columns and rows of the current submatrix as \textit{uncrossed} and to all others as \textit{crossed}. When we say we cross out something, we mean we do it according to the algorithm, that is, we always cross out all remaining occurrences of a variable we chose.

\begin{lemma}
\label{unique-enough}
If there is a unique row monomial of a $k\times \ell$ matrix $M_0$ over an infinite field, then $M_0$ has generically full row rank, that is rank $k$.
\end{lemma}
\begin{proof}
    For each $i$, we take every $a_i$ in $M_{i-1}$. They are in different rows and columns and there is $k$ of them, so they form a monomial in a certain $k \times k$ minor of $M_0$. By construction, there is no other way to get this monomial in any minor. Thus, $M_0$ has a $k\times k$ minor which is non-zero as a polynomial, which in an infinite field means that it is non-zero on an open set of values of $a_i$.
\end{proof}

\subsection{Algorithm}
Now we describe an algorithm to find a unique row monomial, or, in the language of \autoref{cross-out}, to cross out variables so that eventually all rows are crossed out.

\begin{example}
\label{blocks-example}
    We may want to pick an element from the row $(1,2,4)$. Choose for example $a^{1,s}_{2,4}$ in the column $(1,1,s)$ for some $s$. We always have to cross out all remaining occurences of the chosen element. In \autoref{matrix4}, it appears in the row $(3,2,4)$, but with higher $r$ it would also be in rows $(5,2,4)$, $(6,2,4)$ and so on (in columns $(1,5,s)$, $(1,6,s)$ etc). In total, it would be in $r-2$ rows.

    Suppose now $r=5$. With the element $a^{1,s}_{2,4}$, we cross out rows
    
$$\begin{matrix}
 1,2,4 & 3,2,4 & 5,2,4.  
\end{matrix}$$

It will turn out to be useful to loop the row numbers from $1$ to $r$ and consider all rows
$$\begin{matrix}
 1,2,4 & 3,2,4 & 5,2,4 \\  
 2,3,5 & 4,3,5 & 1,3,5 \\
 3,4,1 & 5,4,1 & 2,4,1 \\
 4,5,2 & 1,5,2 & 3,5,2 \\
 5,1,3 & 2,1,3 & 4,1,3,
\end{matrix}$$
and then prove that we can cross them out all.

In the following definition, we will call the three columns “orbits” and the table above a “$1$-block of size 3”. The $1$ means that $t=1$, so the first number changes between orbits, while the two others do not.

%Observe that an orbit is characterised by pairwise differences of its elements modulo $r$, i.e.\,for $(i,j,k)\in \Prr$ we write $(k-j,i-k,j-i)$ modulo $r$ and this will be obviously the same for the whole orbit. For our three orbits, we get
%$$\begin{matrix}
% 2,2,1 & 2,4,4 & 2,1,2.  
%\end{matrix}$$

%We see that the fist number is the same for all three orbits in the block. This is no surprise, because the second and the third number are the same in all of them.

%Finally, we note that the pairwise differences written above sum to $r=5$ for two orbits and to $2r=10$ for one. This is due to the fact that $1,2,4$ and $5,2,4$ are in correct cyclic order whereas $3,2,4$ is backwards. So in the case of the middle orbit or $3,2,4$, we can write the differences negative (modulo $r$) as $-3,-1,-1$. Then they sum to $-r=-5$.
\end{example}

Now we formalize these ideas.

%\begin{lemma}
%\label{pairwise}
%    For any $(i,j,k) \in \Prr$ we consider the triple of pairwise differences modulo $r$: $(a,b,c) := (k-j,i-k,j-i) \in \{0, \ldots, r-1\}$. Then $a,b,c\not=0$ and either $a+b+c=r$ or $a+b+c=2r$.
%\end{lemma}
%\begin{proof}
%yeah, trivial, sure
%    $i,j,k$ are all different by definition of $\Prr$, so no pairwise difference can be zero. $a+b+c = (k-j)+(i-k)+(j-i)$ is zero modulo $r$, so as an integer, it is a multiple of $r$. Since $1\leq a,b,c<r$, the only options are $r$ and $2r$.
%\end{proof}

\begin{definition}
    There is a natural $\mathbb{Z}/r\mathbb{Z}$-action on the set $\Prr$, which we identify with rows of $M$, given by $a \cdot (i,j,k) = (i+a, j+a, k+a) \mod r.$ For $t = 1,2,3$ we say that a set $B$ of $b$ orbits under this action is a \highlight{t-block of size b} if for any two orbits $O_1, O_2 \in B$ and some (or, equivalently, any) $p_1 \in O_1$ there exists $p_2 \in O_2$ such that $p_1$ and $p_2$ differ only in the $t$-th place, e.g.\,$p_1 = (i_1,j,k)$ and $p_2 = (i_2,j,k)$ for $t=1$. Note that $0 \leq b \leq r-2$.

    %We recall \autoref{pairwise}. For any $(i,j,k) \in \Prr$ we consider the triple of pairwise differences modulo $r$: $(a,b,c) := (k-j,i-k,j-i) \in \{1, \ldots, r-1\}$.  In the case where $a+b+c=r$ we define the \textbf{difference triple} of $(i,j,k)$ as $(a,b,c) \in \{1,\ldots, r-1\}$ and in the case where $a+b+c=2r$ we define it as $(a-r,b-r,c-r) \in \{-1,\ldots -(r-1)\}$. We call a triple $(i,j,k) \in \Prr$ \textbf{positive} if its difference triple is of the former type and \textbf{negative} otherwise. 
\end{definition}

%In the lemma below we collect simple facts involving the notions we have introduced above.

%\begin{lemma}
%\begin{itemize}
    %\item[1)] There is a bijection between the sets of positive and negative triples;
    %\item[2)] Every two elements of the same orbit have the same difference triple;
    %\item[3)] There is a bijection between the set of all possible difference triples (triples of positive integers that sum to $r$ and their negative counterparts) and the set of all orbits.
    %\item[4)] All difference triples of elements of an $t$-block of size $k$ share the $t$-th number.
    %\item[5)] Any $t$-block of size $r-2$ is determined by a number $q \in \{-r+1,\ldots,r-1\}\setminus \{0\}$ and a position $t$.
    %I guess we need more
%\end{itemize}
    
%\end{lemma}

We also point out this simple property which will be useful later.

\begin{lemma}
\label{intersection}
If $t_1\not=t_2$, then a $t_1$-block and a $t_2$-block share at most one orbit.
\end{lemma}
\begin{proof}
Note that there is a unique $t\not = t_1,t_2$. From each orbit we can pick a unique representative which has a 1 on the $t$-th place. Assume the two blocks share at least one orbit, whose representative is $p$. Then all other representatives of (orbits of) the $t_1$-block can be obtained by changing the $t_1$-th number of $p$, and similarly, other representatives of the $t_2$-block can be obtained by changing the $t_2$-th number in $p$. Changing the $t_1$-th and $t_2$-th number clearly gives different results, so the blocks share no other orbit.
\end{proof}
 
In our algorithm, we will always cross out whole orbits. This leads us to a notion of a crossed orbit.

\begin{definition}
\label{defblock}
    Assume that elements (rows) of every orbit are either all crossed or all uncrossed. Then we call an orbit \highlight{crossed} if every (or equivalently, any) row in the orbit is crossed. We say a $t$-block is \highlight{maximal uncrossed} if it contains only uncrossed orbits and is maximal among those with respect to inclusion.
\end{definition}

In \autoref{blocks-example}, maximal meant that we had to cross out all three rows that contained an uncrossed variable $a^{1,s}_{2,4}$ and then similarly for the rest of the orbit. However, if a row (with its orbit) is already crossed, $a^{1,s}_{2,4}$ in that row is no longer in the relevant submatrix, so we do not cross it and do not include that row (and orbit) in the block.

Now we have all notions to prove the main result. Fist, we introduce the method to cross out blocks.

\begin{theorem}
\label{block}
    Let $t\in\{1,2,3\}$. Consider a maximal uncrossed $t$-block of size $b$ and let $B$ be the set of rows in the union of the orbits of $B$. Let $S\subseteq [n_t-r]$ be a set of cardinality $b$ and let $C=\{(t,m,s)\mid m\in[r], s\in S\}$ be a set of uncrossed columns.\footnote{The fact that no rows are crossed in $B$ is included in the \autoref{defblock}.} Then it is possible to cross out precisely all rows in $B$ and all columns in $C$ so that in each step when a variable $\Tilde{a}$ is crossed out then all its occurrences are crossed out as well.
\end{theorem}

Note that $|B|=|C|$, which is clearly necessary. Remember that whenever we cross out a variable, we must cross out all its occurrences. Let us first state a lemma.

\begin{lemma}
\label{all-in-block}
Let $\Tilde{a}$ be a variable in the matrix. Assume there is a row in $B$ containing $\Tilde{a}$ and there is a column in $C$ containing $\Tilde{a}$. Then all occurrences of $\Tilde{a}$ in the matrix lie in the rows of $B$ and in the columns of $C$.
\end{lemma}
\begin{proof}
We can assume $t=1$ since other cases are analogous. Note that an element $a^{1,s}_{j,k}$ lies precisely in the rows indexed by $(m,j,k)$ and the columns indexed by $(1,m,s)$ with $m\in[r]\setminus \{j,k\}$. If $C$ contains one of such columns then it contains all of them. On the other hand, if one of the rows is in $B$, then, since $B$ is a maximal uncrossed 1-block, all rows obtained by changing $m$ are either in an orbit of the block, or in a crossed orbit.
\end{proof}

The proof of \autoref{block} will allow us to cross out variables in an arbitrary order. 
%We will see later that we can freely choose the next block to cross out as well.

\begin{proof}[Proof of \autoref{block}]
We pick any variable that lies in an uncrossed row of $B$ and an uncrossed column of $C$ and cross out all rows and columns where it occurs. By \autoref{all-in-block}, we will cross out only allowed rows and columns. Then we pick another variable with the same properties and repeat the process.
To complete the proof we show that as long as there is an uncrossed column in $C$ (or, equivalently, an uncrossed row in $B$ since $|B|=|C|$), there is always an uncrossed element in this column, which we can pick.

We prove it for $t=1$ since other cases are analogous. Assume, for contradiction, that a column $(1,m,s)\in C$ is uncrossed but all non-zero entries of the column lying in rows of $B$ are crossed (note that this can happen only if their corresponding rows are crossed). These entries are $a^{1,s}_{j,k}$ in the row $(m,j,k)$ for all $j,k$ such that $(m,j,k)\in B$. In every orbit in $B$, there is exactly one index triple with the first index equal to $m$, so there is $b$ of them in $B$. If all these rows are crossed then each of them contains an element that has been crossed out. But note that every such element lies in a column $(1,m,s')$ for some $s'\in S$. Now, since no two elements in one column can be crossed out, we see that $b$ columns of the type $(1,m,s')$ in $C$ were crossed out. Since there are precisely $b$ columns of this form in $C$, the column $(1,m,s)$ must have been crossed out as well, which is a contradiction.

\end{proof}

\begin{theorem}
\label{supertheorem}
    For any $n_1,n_2,n_3,r\in \mathbb{N}$ such that $r \leq n_1,n_2,n_3,\sqrt{n_1+n_2+n_3-2}$ the matrix $M$ defined in \ref{matrix} has generic rank\footnote{Sometimes we call it simply “rank”.} $r(r-1)(r-2)$ over any infinite field.
\end{theorem}

\begin{proof}

First note that the assumption $r\leq\sqrt{n_1+n_2+n_3-2}$ implies that $M$ has at least as many columns as rows. That is necessary to have full row rank, and here we prove that is sufficient.

Our strategy will be to find a monomial in $a_{i,j}^{t,s}$'s which appears in an $r(r-1)(r-2)\times r(r-1)(r-2)$ minor of the matrix $M$ with a nonzero coefficient, as in \autoref{unique}. %Since $a_{i,j}^{t,s}$ are generic it follows that this minor is non-zero, hence $M$ has rank $r(r-1)(r-2)$. Note that it is enough to find a minor and a monomial appearing in the minor precisely once. 

We will repeatedly cross out maximal uncrossed blocks as described in \autoref{block}. Note that as long as there are uncrossed orbits, there are (non-empty) maximal uncrossed blocks. By \autoref{block}, we can cross out any such block provided there is a suitable set $C$ of uncrossed columns. Below we will show that at any point there exists a maximal uncrossed block and a suitable set $C$ for this block. This will imply that we can cross out maximal uncrossed blocks according to the algorithm until all orbits are crossed out.

Note that by repeatedly using the algorithm described in the proof of \autoref{block}, we always cross out all $r$ columns $(t,m,s)$ with a given $t$ and $s$. Thus, it makes sense to consider for each $t=1,2,3$ the number of remaining indices $s$ such that (all) columns $(t,m,s)$ for $m\in[r]$ are uncrossed. We denote this number by $\mathcal{S}_t$. Note that if for some $t$ we have that $\mathcal{S}_t$ is at least the size of a maximal $t$-block, then we can cross out this block by \autoref{block} (we can pick a big enough set $S$ of remaining indices). By contradiction, we prove that this is always the case.

Pick an uncrossed orbit $\mathcal{O}$. For each $t=1,2,3$, denote by $B_t$ the unique maximal uncrossed $t$-block containing $\mathcal{O}$. We denote the number of \textit{orbits} in $B_t$ by $|B_t|_o$. Assume for contradiction that for every $t$ we have $|B_t|_o\geq \mathcal{S}_t+1$. Now, let us count the number of orbits in the union $B_1\cup B_2 \cup B_3$. Note that $B_i$'s share one orbit $\mathcal{O}$, so by \autoref{intersection}, no two of them have other common orbits. %(meaning that $|B_i \cup B_j| = |B_i| + |B_j| - 1$ for $i \neq j$). %this is true but not enough I think, so I find it misleading. maybe reformulate it so that it is clear that we don't use this to prove the inequality below
Hence we can calculate
$$|B_1\cup B_2 \cup B_3|_o = |B_1|_o + |B_2|_o + |B_3|_o - 2 \geq \mathcal{S}_1+\mathcal{S}_2+\mathcal{S}_3+1.$$

Counting the rows from orbits in $B_1\cup B_2 \cup B_3$ we see that there are at least $r \cdot |B_1\cup B_2 \cup B_3|_o$ uncrossed rows and exactly $r \cdot (\mathcal{S}_1+\mathcal{S}_2+\mathcal{S}_3)$ uncrossed columns. The inequality above implies that there are more uncrossed rows than uncrossed columns left, which is impossible. Indeed, since $M$ has at least as many columns as rows and, since we always cross out the same number of columns as rows, we can conclude that at any point there are at most as many uncrossed rows as uncrossed columns.

Thus, there is always a $t\in \{1,2,3\}$ for which we can cross out $B_t$ according to \autoref{block}. Therefore, we can find an $r(r-1)(r-2) \times r(r-1)(r-2)$ minor containing a monomial with coefficient $\pm 1$ using the crossing out procedure, and we can use \autoref{unique-enough}.

\end{proof}

\subsection{Conclusion}

\MainCor

\begin{proof}
    %The inequality $Q(n) \leq \lfloor \sqrt{3n-2} \rfloor$ follows from \cite[Theorem 1.2]{10.4230/LIPIcs.CCC.2022.9}. 
    
    Let $r = \min\{n_1,n_2,n_3,\lfloor \sqrt{n_1+n_2+n_3-2} \rfloor\}$. This guarantees $r\leq n_1,n_2,n_3$ so that $M$ to be well-defined and the assumptions of \autoref{thmgen} are satisfied.\footnote{The \autoref{thmgen} also requires $n_i -r \leq r^2$ for $i=1,2,3$. However, if $n_i - r > r^2$ for some $i$, then 
    it is easy to see that $Q(n_1,n_2,n_3) \geq r$.}
    It follows from \autoref{transfer} and \autoref{supertheorem} that if $\mathcal{X}_i \subseteq K^{r,r}$ is a generic subspace of dimension $n_i-r$ for $i=1,2,3$ then $$\mathcal{X}_1[1] + \mathcal{X}_2[2] + \mathcal{X}_3[3] + W_r = K^{r,r,r}.$$ Then by \autoref{thmgen} we have $$Q(n_1,n_2,n_3) \geq r.$$

    The other inequality follows from \cite[Theorem 1.5]{10.4230/LIPIcs.CCC.2022.9}, while $Q(n_1,n_2,n_3)\leq n_1,n_2,n_3$ is known for any tensor.
    %For the other inequality, pick any $r > \min\{n_1,n_2,n_3,\lfloor \sqrt{n_1+n_2+n_3-2} \rfloor\}$. Then either $r > n_i$ for some $i$ or $r>\sqrt{n_1+n_2+n_3-2}$. In the former case we trivially have $Q(n_1,n_2,n_3) < r$ while in the latter $M$ has less columns than rows so, by the converse of Theorem $\ref{thmgen}$, we have $Q(n_1,n_2,n_3) < r$.
\end{proof}

\SimpleMainCor

\begin{remark}
\label{equivalence-remark}
    \autoref{supertheorem} says that $M$ has full row rank whenever it has at least as many columns as rows. Thus, it is easy to see from our method that $Q(n_1,n_2,n_3) \leq \lfloor \sqrt{n_1+n_2+n_3-2} \rfloor$ if the field has characteristic zero. Indeed, for $r>\sqrt{n_1+n_2+n_3-2}$ the matrix has fewer columns than rows, so it is not of full row rank and we can use the converse implication in \autoref{thmgen}.
\end{remark}

\begin{example}
%todo do we keep it? also indices of rows and columns would actually be nice
To intrigued readers we give the result of our algorithm when applied on the matrix from $\autoref{matrix4}$, that is with $r=4$ and $n_1=n_2=n_3=6$. Note that there are many more ways to perform it because \autoref{all-in-block} and \autoref{supertheorem} leave some freedom of choice.
    \begingroup\small$$\hskip-2cm\begin{pNiceMatrix}[create-medium-nodes, cell-space-limits = 0.9pt]
a^{11}_{23} & a^{12}_{23} & & & & & & & & & a^{21}_{13} & a^{22}_{13} & & & & & & & & & a^{31}_{12} & a^{32}_{12} & & \\
a^{11}_{24} & a^{12}_{24} & & & & & & & & & a^{21}_{14} & a^{22}_{14} & & & & & & & & & & & a^{31}_{12} & a^{32}_{12} \\
a^{11}_{32} & a^{12}_{32} & & & & & & & & & & & a^{21}_{12} & a^{22}_{12} & & & & & a^{31}_{13} & a^{32}_{13} & & & & \\
a^{11}_{34} & a^{12}_{34} & & & & & & & & & & & a^{21}_{14} & a^{22}_{14} & & & & & & & & & a^{31}_{13} & a^{32}_{13} \\
a^{11}_{42} & a^{12}_{42} & & & & & & & & & & & & & a^{21}_{12} & a^{22}_{12} & & & a^{31}_{14} & a^{32}_{14} & & & & \\
a^{11}_{43} & a^{12}_{43} & & & & & & & & & & & & & a^{21}_{13} & a^{22}_{13} & & & & & a^{31}_{14} & a^{32}_{14} & & \\
& & a^{11}_{13} & a^{12}_{13} & & & & & a^{21}_{23} & a^{22}_{23} & & & & & & & & & & & a^{31}_{21} & a^{32}_{21} & & \\
& & a^{11}_{14} & a^{12}_{14} & & & & & a^{21}_{24} & a^{22}_{24} & & & & & & & & & & & & & a^{31}_{21} & a^{32}_{21} \\
& & a^{11}_{31} & a^{12}_{31} & & & & & & & & & a^{21}_{21} & a^{22}_{21} & & & a^{31}_{23} & a^{32}_{23} & & & & & & \\
& & a^{11}_{34} & a^{12}_{34} & & & & & & & & & a^{21}_{24} & a^{22}_{24} & & & & & & & & & a^{31}_{23} & a^{32}_{23} \\
& & a^{11}_{41} & a^{12}_{41} & & & & & & & & & & & a^{21}_{21} & a^{22}_{21} & a^{31}_{24} & a^{32}_{24} & & & & & & \\
& & a^{11}_{43} & a^{12}_{43} & & & & & & & & & & & a^{21}_{23} & a^{22}_{23} & & & & & a^{31}_{24} & a^{32}_{24} & & \\
& & & & a^{11}_{12} & a^{12}_{12} & & & a^{21}_{32} & a^{22}_{32} & & & & & & & & & a^{31}_{31} & a^{32}_{31} & & & & \\
& & & & a^{11}_{14} & a^{12}_{14} & & & a^{21}_{34} & a^{22}_{34} & & & & & & & & & & & & & a^{31}_{31} & a^{32}_{31} \\
& & & & a^{11}_{21} & a^{12}_{21} & & & & & a^{21}_{31} & a^{22}_{31} & & & & & a^{31}_{32} & a^{32}_{32} & & & & & & \\
& & & & a^{11}_{24} & a^{12}_{24} & & & & & a^{21}_{34} & a^{22}_{34} & & & & & & & & & & & a^{31}_{32} & a^{32}_{32} \\
& & & & a^{11}_{41} & a^{12}_{41} & & & & & & & & & a^{21}_{31} & a^{22}_{31} & a^{31}_{34} & a^{32}_{34} & & & & & & \\
& & & & a^{11}_{42} & a^{12}_{42} & & & & & & & & & a^{21}_{32} & a^{22}_{32} & & & a^{31}_{34} & a^{32}_{34} & & & & \\
& & & & & & a^{11}_{12} & a^{12}_{12} & a^{21}_{42} & a^{22}_{42} & & & & & & & & & a^{31}_{41} & a^{32}_{41} & & & & \\
& & & & & & a^{11}_{13} & a^{12}_{13} & a^{21}_{43} & a^{22}_{43} & & & & & & & & & & & a^{31}_{41} & a^{32}_{41} & & \\
& & & & & & a^{11}_{21} & a^{12}_{21} & & & a^{21}_{41} & a^{22}_{41} & & & & & a^{31}_{42} & a^{32}_{42} & & & & & & \\
& & & & & & a^{11}_{23} & a^{12}_{23} & & & a^{21}_{43} & a^{22}_{43} & & & & & & & & & a^{31}_{42} & a^{32}_{42} & & \\
& & & & & & a^{11}_{31} & a^{12}_{31} & & & & & a^{21}_{41} & a^{22}_{41} & & & a^{31}_{43} & a^{32}_{43} & & & & & & \\
& & & & & & a^{11}_{32} & a^{12}_{32} & & & & & a^{21}_{42} & a^{22}_{42} & & & & & a^{31}_{43} & a^{32}_{43} & & & & \cdot
\CodeAfter
\begin{tikzpicture}[line width=0.3mm, rounded corners=1mm]
\node (1-1r) at (1-1) [draw,minimum width=0.6cm,minimum height=0.45cm] {};
\draw (1-1r.east) -- (1-24-medium.east) ;
\draw (2-1-medium.north) -- (24-1-medium.south) ;
\node (22-7r) at (22-7) [draw,minimum width=0.6cm,minimum height=0.45cm] {};
\draw (22-1-medium.west) -- (22-7r.west) ;
\draw (22-7r.east) -- (22-24-medium.east) ;
\draw (1-7-medium.north) -- (21-7-medium.south) ;
\draw (23-7-medium.north) -- (24-7-medium.south) ;
\node (10-3r) at (10-3) [draw,minimum width=0.6cm,minimum height=0.45cm] {};
\draw (10-2-medium.west) -- (10-3r.west) ;
\draw (10-3r.east) -- (10-6-medium.east) ;
\draw (10-8-medium.west) -- (10-24-medium.east) ;
\draw (2-3-medium.north) -- (9-3-medium.south) ;
\draw (11-3-medium.north) -- (21-3-medium.south) ;
\draw (23-3-medium.north) -- (24-3-medium.south) ;
\node (17-5r) at (17-5) [draw,minimum width=0.6cm,minimum height=0.45cm] {};
\draw (17-2-medium.west) -- (17-2-medium.east) ;
\draw (17-4-medium.west) -- (17-5r.west) ;
\draw (17-5r.east) -- (17-6-medium.east) ;
\draw (17-8-medium.west) -- (17-24-medium.east) ;
\draw (2-5-medium.north) -- (9-5-medium.south) ;
\draw (11-5-medium.north) -- (16-5-medium.south) ;
\draw (18-5-medium.north) -- (21-5-medium.south) ;
\draw (23-5-medium.north) -- (24-5-medium.south) ;
\node (4-2r) at (4-2) [draw,minimum width=0.6cm,minimum height=0.45cm] {};
\draw (4-4-medium.west) -- (4-4-medium.east) ;
\draw (4-6-medium.west) -- (4-6-medium.east) ;
\draw (4-8-medium.west) -- (4-24-medium.east) ;
\draw (2-2-medium.north) -- (3-2-medium.south) ;
\draw (5-2-medium.north) -- (9-2-medium.south) ;
\draw (11-2-medium.north) -- (16-2-medium.south) ;
\draw (18-2-medium.north) -- (21-2-medium.south) ;
\draw (23-2-medium.north) -- (24-2-medium.south) ;
\node (11-4r) at (11-4) [draw,minimum width=0.6cm,minimum height=0.45cm] {};
\draw (11-6-medium.west) -- (11-6-medium.east) ;
\draw (11-8-medium.west) -- (11-24-medium.east) ;
\draw (2-4-medium.north) -- (3-4-medium.south) ;
\draw (5-4-medium.north) -- (9-4-medium.south) ;
\draw (12-4-medium.north) -- (16-4-medium.south) ;
\draw (18-4-medium.north) -- (21-4-medium.south) ;
\draw (23-4-medium.north) -- (24-4-medium.south) ;
\node (13-6r) at (13-6) [draw,minimum width=0.6cm,minimum height=0.45cm] {};
\draw (13-8-medium.west) -- (13-24-medium.east) ;
\draw (2-6-medium.north) -- (3-6-medium.south) ;
\draw (5-6-medium.north) -- (9-6-medium.south) ;
\draw (12-6-medium.north) -- (12-6-medium.south) ;
\draw (14-6-medium.north) -- (16-6-medium.south) ;
\draw (18-6-medium.north) -- (21-6-medium.south) ;
\draw (23-6-medium.north) -- (24-6-medium.south) ;
\node (19-8r) at (19-8) [draw,minimum width=0.6cm,minimum height=0.45cm] {};
\draw (19-6-medium.west) -- (19-6-medium.east) ;
\draw (19-8r.east) -- (19-24-medium.east) ;
\draw (2-8-medium.north) -- (3-8-medium.south) ;
\draw (5-8-medium.north) -- (9-8-medium.south) ;
\draw (12-8-medium.north) -- (16-8-medium.south) ;
\draw (18-8-medium.north) -- (18-8-medium.south) ;
\draw (20-8-medium.north) -- (21-8-medium.south) ;
\draw (23-8-medium.north) -- (24-8-medium.south) ;
\node (20-9r) at (20-9) [draw,minimum width=0.6cm,minimum height=0.45cm] {};
\draw (20-9r.east) -- (20-24-medium.east) ;
\draw (2-9-medium.north) -- (3-9-medium.south) ;
\draw (5-9-medium.north) -- (9-9-medium.south) ;
\draw (12-9-medium.north) -- (12-9-medium.south) ;
\draw (14-9-medium.north) -- (16-9-medium.south) ;
\draw (18-9-medium.north) -- (18-9-medium.south) ;
\draw (21-9-medium.north) -- (21-9-medium.south) ;
\draw (23-9-medium.north) -- (24-9-medium.south) ;
\node (2-11r) at (2-11) [draw,minimum width=0.6cm,minimum height=0.45cm] {};
\draw (2-10-medium.west) -- (2-11r.west) ;
\draw (2-11r.east) -- (2-24-medium.east) ;
\draw (3-11-medium.north) -- (3-11-medium.south) ;
\draw (5-11-medium.north) -- (9-11-medium.south) ;
\draw (12-11-medium.north) -- (12-11-medium.south) ;
\draw (14-11-medium.north) -- (16-11-medium.south) ;
\draw (18-11-medium.north) -- (18-11-medium.south) ;
\draw (21-11-medium.north) -- (21-11-medium.south) ;
\draw (23-11-medium.north) -- (24-11-medium.south) ;
\node (9-13r) at (9-13) [draw,minimum width=0.6cm,minimum height=0.45cm] {};
\draw (9-10-medium.west) -- (9-10-medium.east) ;
\draw (9-12-medium.west) -- (9-13r.west) ;
\draw (9-13r.east) -- (9-24-medium.east) ;
\draw (3-13-medium.north) -- (3-13-medium.south) ;
\draw (5-13-medium.north) -- (8-13-medium.south) ;
\draw (12-13-medium.north) -- (12-13-medium.south) ;
\draw (14-13-medium.north) -- (16-13-medium.south) ;
\draw (18-13-medium.north) -- (18-13-medium.south) ;
\draw (21-13-medium.north) -- (21-13-medium.south) ;
\draw (23-13-medium.north) -- (24-13-medium.south) ;
\node (18-15r) at (18-15) [draw,minimum width=0.6cm,minimum height=0.45cm] {};
\draw (18-10-medium.west) -- (18-10-medium.east) ;
\draw (18-12-medium.west) -- (18-12-medium.east) ;
\draw (18-14-medium.west) -- (18-15r.west) ;
\draw (18-15r.east) -- (18-24-medium.east) ;
\draw (3-15-medium.north) -- (3-15-medium.south) ;
\draw (5-15-medium.north) -- (8-15-medium.south) ;
\draw (12-15-medium.north) -- (12-15-medium.south) ;
\draw (14-15-medium.north) -- (16-15-medium.south) ;
\draw (21-15-medium.north) -- (21-15-medium.south) ;
\draw (23-15-medium.north) -- (24-15-medium.south) ;
\node (21-17r) at (21-17) [draw,minimum width=0.6cm,minimum height=0.45cm] {};
\draw (21-10-medium.west) -- (21-10-medium.east) ;
\draw (21-12-medium.west) -- (21-12-medium.east) ;
\draw (21-14-medium.west) -- (21-14-medium.east) ;
\draw (21-16-medium.west) -- (21-17r.west) ;
\draw (21-17r.east) -- (21-24-medium.east) ;
\draw (3-17-medium.north) -- (3-17-medium.south) ;
\draw (5-17-medium.north) -- (8-17-medium.south) ;
\draw (12-17-medium.north) -- (12-17-medium.south) ;
\draw (14-17-medium.north) -- (16-17-medium.south) ;
\draw (23-17-medium.north) -- (24-17-medium.south) ;
\node (3-19r) at (3-19) [draw,minimum width=0.6cm,minimum height=0.45cm] {};
\draw (3-10-medium.west) -- (3-10-medium.east) ;
\draw (3-12-medium.west) -- (3-12-medium.east) ;
\draw (3-14-medium.west) -- (3-14-medium.east) ;
\draw (3-16-medium.west) -- (3-16-medium.east) ;
\draw (3-18-medium.west) -- (3-19r.west) ;
\draw (3-19r.east) -- (3-24-medium.east) ;
\draw (5-19-medium.north) -- (8-19-medium.south) ;
\draw (12-19-medium.north) -- (12-19-medium.south) ;
\draw (14-19-medium.north) -- (16-19-medium.south) ;
\draw (23-19-medium.north) -- (24-19-medium.south) ;
\node (12-21r) at (12-21) [draw,minimum width=0.6cm,minimum height=0.45cm] {};
\draw (12-10-medium.west) -- (12-10-medium.east) ;
\draw (12-12-medium.west) -- (12-12-medium.east) ;
\draw (12-14-medium.west) -- (12-14-medium.east) ;
\draw (12-16-medium.west) -- (12-16-medium.east) ;
\draw (12-18-medium.west) -- (12-18-medium.east) ;
\draw (12-20-medium.west) -- (12-21r.west) ;
\draw (12-21r.east) -- (12-24-medium.east) ;
\draw (5-21-medium.north) -- (8-21-medium.south) ;
\draw (14-21-medium.north) -- (16-21-medium.south) ;
\draw (23-21-medium.north) -- (24-21-medium.south) ;
\node (14-23r) at (14-23) [draw,minimum width=0.6cm,minimum height=0.45cm] {};
\draw (14-10-medium.west) -- (14-10-medium.east) ;
\draw (14-12-medium.west) -- (14-12-medium.east) ;
\draw (14-14-medium.west) -- (14-14-medium.east) ;
\draw (14-16-medium.west) -- (14-16-medium.east) ;
\draw (14-18-medium.west) -- (14-18-medium.east) ;
\draw (14-20-medium.west) -- (14-20-medium.east) ;
\draw (14-22-medium.west) -- (14-23r.west) ;
\draw (14-23r.east) -- (14-24-medium.east) ;
\draw (5-23-medium.north) -- (8-23-medium.south) ;
\draw (15-23-medium.north) -- (16-23-medium.south) ;
\draw (23-23-medium.north) -- (24-23-medium.south) ;
\node (7-10r) at (7-10) [draw,minimum width=0.6cm,minimum height=0.45cm] {};
\draw (7-12-medium.west) -- (7-12-medium.east) ;
\draw (7-14-medium.west) -- (7-14-medium.east) ;
\draw (7-16-medium.west) -- (7-16-medium.east) ;
\draw (7-18-medium.west) -- (7-18-medium.east) ;
\draw (7-20-medium.west) -- (7-20-medium.east) ;
\draw (7-22-medium.west) -- (7-22-medium.east) ;
\draw (7-24-medium.west) -- (7-24-medium.east) ;
\draw (5-10-medium.north) -- (6-10-medium.south) ;
\draw (8-10-medium.north) -- (8-10-medium.south) ;
\draw (15-10-medium.north) -- (16-10-medium.south) ;
\draw (23-10-medium.north) -- (24-10-medium.south) ;
\node (16-12r) at (16-12) [draw,minimum width=0.6cm,minimum height=0.45cm] {};
\draw (16-14-medium.west) -- (16-14-medium.east) ;
\draw (16-16-medium.west) -- (16-16-medium.east) ;
\draw (16-18-medium.west) -- (16-18-medium.east) ;
\draw (16-20-medium.west) -- (16-20-medium.east) ;
\draw (16-22-medium.west) -- (16-22-medium.east) ;
\draw (16-24-medium.west) -- (16-24-medium.east) ;
\draw (5-12-medium.north) -- (6-12-medium.south) ;
\draw (8-12-medium.north) -- (8-12-medium.south) ;
\draw (15-12-medium.north) -- (15-12-medium.south) ;
\draw (23-12-medium.north) -- (24-12-medium.south) ;
\node (23-14r) at (23-14) [draw,minimum width=0.6cm,minimum height=0.45cm] {};
\draw (23-16-medium.west) -- (23-16-medium.east) ;
\draw (23-18-medium.west) -- (23-18-medium.east) ;
\draw (23-20-medium.west) -- (23-20-medium.east) ;
\draw (23-22-medium.west) -- (23-22-medium.east) ;
\draw (23-24-medium.west) -- (23-24-medium.east) ;
\draw (5-14-medium.north) -- (6-14-medium.south) ;
\draw (8-14-medium.north) -- (8-14-medium.south) ;
\draw (15-14-medium.north) -- (15-14-medium.south) ;
\draw (24-14-medium.north) -- (24-14-medium.south) ;
\node (5-16r) at (5-16) [draw,minimum width=0.6cm,minimum height=0.45cm] {};
\draw (5-18-medium.west) -- (5-18-medium.east) ;
\draw (5-20-medium.west) -- (5-20-medium.east) ;
\draw (5-22-medium.west) -- (5-22-medium.east) ;
\draw (5-24-medium.west) -- (5-24-medium.east) ;
\draw (6-16-medium.north) -- (6-16-medium.south) ;
\draw (8-16-medium.north) -- (8-16-medium.south) ;
\draw (15-16-medium.north) -- (15-16-medium.south) ;
\draw (24-16-medium.north) -- (24-16-medium.south) ;
\node (15-18r) at (15-18) [draw,minimum width=0.6cm,minimum height=0.45cm] {};
\draw (15-20-medium.west) -- (15-20-medium.east) ;
\draw (15-22-medium.west) -- (15-22-medium.east) ;
\draw (15-24-medium.west) -- (15-24-medium.east) ;
\draw (6-18-medium.north) -- (6-18-medium.south) ;
\draw (8-18-medium.north) -- (8-18-medium.south) ;
\draw (24-18-medium.north) -- (24-18-medium.south) ;
\node (24-20r) at (24-20) [draw,minimum width=0.6cm,minimum height=0.45cm] {};
\draw (24-22-medium.west) -- (24-22-medium.east) ;
\draw (24-24-medium.west) -- (24-24-medium.east) ;
\draw (6-20-medium.north) -- (6-20-medium.south) ;
\draw (8-20-medium.north) -- (8-20-medium.south) ;
\node (6-22r) at (6-22) [draw,minimum width=0.6cm,minimum height=0.45cm] {};
\draw (6-24-medium.west) -- (6-24-medium.east) ;
\draw (8-22-medium.north) -- (8-22-medium.south) ;
\node (8-24r) at (8-24) [draw,minimum width=0.6cm,minimum height=0.45cm] {};
\end{tikzpicture}
\end{pNiceMatrix}
$$
\endgroup

The unique row monomial is
$$\left(a^{1,1}_{2,3}\right)^2\, a^{1,1}_{3,4}\, a^{1,1}_{4,1}\, a^{1,2}_{3,4}\, a^{1,2}_{4,1}\, \left(a^{1,2}_{1,2}\right)^2\, a^{2,1}_{4,3}\, a^{2,1}_{1,4}\, a^{2,1}_{2,1}\, a^{2,1}_{3,2}\, a^{3,1}_{4,2}\, a^{3,1}_{1,3}\, a^{3,1}_{2,4}\, a^{3,1}_{3,1}\, a^{2,2}_{2,3}\, a^{2,2}_{3,4}\, a^{2,2}_{4,1}\, a^{2,2}_{1,2}\, a^{3,2}_{3,2}\, a^{3,2}_{4,3}\, a^{3,2}_{1,4}\, a^{3,2}_{2,1}.$$

First we cross out a maximal uncrossed $1$-block of size 2 which contains orbits of rows $(1,2,3)$ and $(1,3,4)$ using $S=\{1,2\}$. Then using $S=\{1\}$ a maximal uncrossed $2$-block of size 1 which contains only the orbit of $(1,2,4)$ -- the orbit of $(1,3,4)$ is already crossed.

Then a $3$-block with $(1,3,2)$ using $S=\{1\}$, a $2$-block with $(1,4,2)$ using $S=\{2\}$ and a $3$-block with $(1,4,3)$ using $S=\{2\}$.
\end{example}

\section{Generic subrank of higher-order tensors}
\label{higher-order}

In this section we will generalise the results of the previous section to higher-order tensors. Below we will state all results in these settings but omit parts of the proofs which are straightforward generalisations of the previous results and of the work in \cite{10.4230/LIPIcs.CCC.2022.9}.

First, note that for any $n_1,\ldots, n_k$ there is a non-empty Zariski-open subset $U \subseteq K^{n_1,\ldots, n_k}$ and integer $r$ such that for all $T \in U$ we have $Q(T)=r$.

%We define $$\mathcal{C}_r = \{T \in K^{n_1, \ldots, n_k} \: | \: Q(T) \geq r\}.$$

%We can parametrize $\mathcal{C}_r$ 
%First consider how does our paramerization map look in the case of $k$-tensors. Let's define:\\
We define $$X_r^{(n_1,\ldots,n_k)} := \{T\in K^{n_1, \ldots, n_k} \: | \: T_{i_1,\dots,i_k}=0 \text{ for } (i_1,\ldots,i_k) \in [r]^k \setminus \{ (i,\ldots,i) \:|\: i \in [r] \}, T_{i,\ldots,i} \neq 0 \text{ for } i \in [r]\}.$$  
%Then $\psi^k_{n_1,\hdots,n_k} : GL_{n_1} \times \hdots \times GL_{n_k} \times X_r \rightarrow K^{n_1} \otimes \hdots \otimes C^{n_k} \rightarrow K^{n_1} \otimes \hdots \otimes K^{n_k}$ has as its image the set of tensors of subrank $\geq r$.\\
%\begin{lemma}
    %The image of the map $\psi^k_{n_1,\hdots,n_k} : GL_{n_1} \times \hdots \times GL_{n_k} \times X_r \rightarrow K^{n_1} \otimes \hdots \otimes C^{n_k} \rightarrow K^{n_1} \otimes \hdots \otimes K^{n_k}$ is the set of matrices of subrank $\geq$ r.
%\end{lemma}
Let $$\Prr^{(k)} := \{(j_1, \ldots, j_k) \in [r]^k \: | \:\: \text{any subset of } (j_1, \ldots, j_k) \text{ of order } k-1 \text{ has at least 2 disjoint elements}\}$$
and
$$W_r^{(k)} = \{T \in K^{r,\ldots, r} | T_{j_1,\hdots,j_k} = 0 \text{ for } (j_1,\hdots,j_k) \in \Prr^{(k)}\}.$$
Note that these definitions generalise the case $k=3$. Indeed, it follows from the proof of \cite[Theorem 3.9]{10.4230/LIPIcs.CCC.2022.9} that tensors in $W_r^{(k)}$ are sums of the flattenings of tensors of $X_r^{(n_1,\ldots,n_k)}$ with respect to one of the coordinates. 
Now we can state the $k$-tensor version of \autoref{thmgen}.
\begin{theorem}
\label{dmz-gen}
    Let $r \leq n_1,\hdots,n_k$ such that $n_i - r \leq r^2$ for $i = 1,\hdots,k$. Let $\mathcal{X}_i \subseteq \bigotimes_{j=1}^{k-1} K^r$ be a generic linear subspace of dimension $n_i-r$ for $i=1,\hdots,k$. Suppose $$\mathcal{X}_1[1] + \hdots + \mathcal{X}_k[k] + W_r^{(k)} = K^{r,\ldots, r}.$$
    Then $Q(n_1,\hdots,n_k) \geq r$. If characteristic of $K$ is 0, then the converse holds.
\end{theorem}
As before, we can translate the condition of the theorem above to the following matrix being full rank:
\begin{lemma}
\label{transfergen}
    With the assumptions of \nameref{dmz-gen} above we have $$\mathcal{X}_1[1] + \hdots + \mathcal{X}_k[k] + W_r^{(k)} = K^{r,\ldots,r},$$ if and only if the matrix $M \in K^{r(n_1+\hdots+n_k - kr) \times (r^k-kr^2+(k-1)r)}$ defined by $$M^{t,m,s}_{j_1,\hdots,j_k} = \begin{cases}
        a^{1,s}_{j_2,\hdots,j_k},& t=1, m=j_1; \\
        a^{2,s}_{j_1,j_3,\hdots,j_k},& t=2, m=j_2; \\
        \dots \\
        a^{k,s}_{j_1,\hdots,j_{k-1}},& t=k, m=j_k; \\
        0 & \text{otherwise}.
    \end{cases}$$
    for generic $a^{t,s}_{\bullet,\hdots,\bullet} \in K$ has rank $r^k-kr^2+(k-1)r$. Here the columns of $M$ are labelled by triples $(t,m,s)$ with $t \in [k],  m \in [r], s \in [n_t-r]$, and the rows by tuples $(j_1,\hdots,j_k) \in \Prr^{(k)}$.
\end{lemma}

We claim that our method to prove \autoref{supertheorem} can be used to show the matrix is of full rank. To make this precise, we define the analogues of blocks and orbits in the generalised context.
\begin{definition}
  There is a natural $\mathbb{Z}/r\mathbb{Z}$-action on the set $\Prr^{(k)}$, which we identify with rows of $M$, given by $a \cdot (j_1,\hdots,j_k) = (j_1+a,\hdots,j_k+a) \mod r.$ For $t = 1,\hdots,k$ we say that a set $B$ of $b$ orbits under this action is an \highlight{t-block of size b} if for any two orbits $O_1, O_2 \in B$ and some (or, equivalently, any) $p_1 \in O_1$ there exists $p_2 \in O_2$ such that $p_1$ and $p_2$ differ only in the $t$-th place.
\end{definition}

The proof of the following lemma is completely analogous to \autoref{intersection}.

\begin{lemma}
\label{intgen}
    Let $k\geq 3$. If $t_1\neq t_2$, then a $t_1$-block and a $t_2$-block (with rows in the set $\Prr^{(k)}$) share at most one orbit.
\end{lemma}

%We point out the case of matrices (when $k=2$). This \nameref{intgen} doesn't hold for them. Still, \autoref{HighOrderMainCor} suggests generic subrank $Q(m,n) = \min \{m,n\}$, which is correct.

We can cross out maximal uncrossed blocks in a way analogous to \autoref{block}.

\begin{theorem}
\label{bigblock}
    Let $t\in\{1,2,\ldots,k\}$. Consider a maximal uncrossed $t$-block of size $b$ and let $B$ be the set of rows in its orbits. Let $S\subseteq [n_t-r]$ of cardinality $b$ and let $C=\{(t,m,s)\mid m\in[r], s\in S\}$ be a set of uncrossed columns. Then it is possible to cross out precisely all rows in $B$ and all columns in $C$.
\end{theorem}

We are ready to state the main theorem of this section, generalising \autoref{supertheorem}.

\begin{theorem}
\label{supermegatheorem}
    Let $k\geq 3$. For any $n_1,\ldots, n_k,r\in \mathbb{N}$ such that $$r \leq n_1,\ldots,n_k,\left\lfloor (n_1+\hdots+n_k-(k-1))^{\frac{1}{k-1}} \right\rfloor$$ the matrix $M$ defined as in \autoref{transfergen} has generic rank $r^k - kr^2 + (k-1)r$ over any infinite field.
\end{theorem}

\begin{proof}
    As in the proof of \autoref{supertheorem}
    we repeatedly cross out maximal uncrossed blocks from the matrix to produce a monomial with coefficient $\pm 1$ in some maximal minor. Since we can cross out blocks using \autoref{bigblock}, we only have to check that there is always a maximal uncrossed block with a suitable set of columns. With notation from the proof of \autoref{supertheorem} it follows from \autoref{intgen} that $$|B_1\cup \hdots \cup B_k|_{o} = |B_1|_{o} + \hdots + |B_k|_{o} - (k-1) \geq \mathcal{S}_1+\dots+\mathcal{S}_k+1.$$ The rest of the proof is analogous to \autoref{supertheorem}.
\end{proof}

\HighOrderMainCor

\begin{proof}
    We apply \autoref{supermegatheorem}. The proof is analogous to \autoref{maincor}.
\end{proof}

\section{Dimension of spaces of tensors of subrank at least $r$}

Recall that in order to bound the generic subrank of tensors in $K^{n_1,\ldots,n_k}$ from above one bounds the dimension of $\mathcal{C}_r^{(n_1,\ldots,n_k)} = \{T \in K^{n_1,\ldots,n_k} \: | \: Q(T) \geq r\}$ (see \cite{10.4230/LIPIcs.CCC.2022.9} for more details). As the spaces $\mathcal{C}_r^{(n_1,\ldots,n_k)}$ are interesting on their own, we apply the methods developed in the previous sections to determine dimensions of these spaces precisely. In this section we often use the notation introduced in the previous section.

To begin with, we have in fact proved in \autoref{supermegatheorem} that if $$r = \min\left\{n_1,\ldots,n_k,\left\lfloor (n_1+\hdots+n_k-(k-1))^{\frac{1}{k-1}} \right\rfloor\right\},$$ then $\dim(\mathcal{C}_r^{(n_1,\ldots,n_k)}) = \dim(K^{n_1,\dots,n_k}) = n_1n_2\dots n_k$. We can generalise this as follows.

\DimensionThm

\begin{proof}
    
    Recall that we have a map $$\psi_r^{(n_1,\ldots,n_k)}: \prod_{i=1}^kGL_{n_i}(K) \times X_r^{(n_1,\ldots,n_k)} \to K^{n_1,\ldots,n_k}$$
    $$(A_1,\ldots, A_{n_k}, T) \mapsto (A_1 \otimes \ldots \otimes A_{n_k})T$$ whose image is precisely $\mathcal{C}_{r}^{(n_1,\ldots,n_k)}$. The proof of this fact is analogous to the proof of \cite[Lemma 2.2]{10.4230/LIPIcs.CCC.2022.9}. By the fiber dimension theorem $$\dim(\mathcal{C}_r^{(n_1,\ldots,n_k)}) = \dim(\prod_{i=1}^kGL_{n_i}(K) \times X_r^{(n_1,\ldots,n_k)}) - \dim(\text{general fiber of }\psi_r^{(n_1,\ldots,n_k)}),$$ hence, in order to bound the dimension of $\mathcal{C}_r^{(n_1,\ldots,n_k)}$ from above one has to bound the dimension of a general fiber of $\psi_r^{(n_1,\ldots,n_k)}$ from below. Following the proofs of \cite[Proposition 2.4 and Theorem 1.5]{10.4230/LIPIcs.CCC.2022.9} we get $$\dim (\mathcal{C}_r^{(n_1,\ldots,n_k)}) \leq \prod_{i=1}^k n_i - r\left(r^{k-1} - \sum_{i=1}^k n_i + (k-1)\right)$$
    
    Recall that the dimension of the image of $\psi_r^{(n_1,\ldots,n_k)}$ is equal to the rank of of the differential of $\psi_r^{(n_1,\ldots,n_k)}$ at a generic point. 
    A straightforward generalisation of the proof of \cite[Theorem 3.9]{10.4230/LIPIcs.CCC.2022.9} shows that the image of a differential of $\psi_r^{(n_1,\ldots,n_k)}$ at a generic point is $$\mathcal{X}_1[1] + \hdots + \mathcal{X}_k[k] + W_r^{(k)} + Y_r^{(n_1,\ldots,n_k)} \subseteq K^{n_1,\ldots,n_k}$$
    where $$Y_r^{(n_1,\ldots,n_k)} = \{T \in K^{n_1,\ldots,n_k} \: | \: T_{j_1,\ldots,j_k} = 0 \text{ for } (j_1,\ldots,j_k) \in [r]^k \setminus \{(i,i,\ldots,i) \: | \: i \in [r]\}\}$$ and $\mathcal{X}_i \subseteq \bigotimes_{j=1}^{k-1} K^r$ is a generic linear subspace of dimension $n_i-r$ for $i=1,\hdots,k$. 

    It follows from the definition that $$\dim(\mathcal{X}_1[1] + \hdots + \mathcal{X}_k[k] + W_r^{(k)} + Y_r^{(n_1,\ldots,n_k)}) =  \dim(Y_r^{(n_1,\ldots,n_k)}) + \dim(\mathcal{X}_1[1] + \hdots + \mathcal{X}_k[k] + W_r^{(k)}) - r.$$
    
    It follows from the proof of \autoref{transfergen} that $$\dim(\mathcal{X}_1[1] + \hdots + \mathcal{X}_k[k] + W_r^{(k)}) = r\left(\sum_{i = 1}^{k}n_i - kr + (kr - k + 1)\right)$$
    if and only if the columns of $M$ defined as in \autoref{transfergen} are linearly independent for a generic set of variables. 
    If this is the case then $$\dim(\mathcal{C}_r^{(n_1,\ldots,n_k)}) = \dim(\im \psi_r^{(n_1,\ldots,n_k)}) = \dim(Y_r^{(n_1,\ldots,n_k)}) + \dim(\mathcal{X}_1[1] + \hdots + \mathcal{X}_k[k] + W_r^{(k)}) - r$$ $$ = \left(\prod_{i=1}^k n_i - r^k + r \right) + r\left(\sum_{i = 1}^{k}n_i - kr+ (kr - k + 1)\right) - r = \prod_{i=1}^k n_i - r\left(r^{k-1} - \sum_{i=1}^k n_i + (k-1)\right).$$
    
    Note that the assumptions imply that $r> (n_1+\hdots+n_k-(k-1))^{\frac{1}{k-1}}$, so the matrix $M$ has more rows than columns. We will prove that the columns of $M$ are linearly independent by adding columns to $M$ to make it a square matrix and using \autoref{supermegatheorem} to show it has full rank.\\
    Suppose that we alter our matrix by increasing $n_t$ for some $t$ by 1. Then we added precisely $r$ columns to the matrix $M$ (indexed by $(t,m,n_t+1)$ for all $m\in[r]$). Since both the number of rows and columns in our original matrix $M$ is divisible by $r$ we can apply such an operation finitely many times to make $M$ a square matrix. Then by \autoref{supermegatheorem} it has full row rank and thus full rank. Then since it is a square matrix, all the columns of the new matrix are linearly independent, hence all columns were linearly independent in the original matrix $M$. 
    %The dimension of the constructible set $\mathcal{C}_r$ is equal to the dimension of its tangent space at a generic point.We parametrize it just as \cite{10.4230/LIPIcs.CCC.2022.9} in, which also leads to \autoref{thmgen}. Its tangent space at a point is the column span of the matrix $M$. We care about the dimension at a generic point, which is them equal to the generic rank of $M$. Here, however, $r\geq Q(n_1,\dots,n_k)$, so $M$ may have more columns than rows and we cannot apply our \nameref{supermegatheorem} about this generic rank.
    %If $M$ is a square matrix, then by \autoref{supermegatheorem} it has full row rank, hence its columns are linearly independent as well. Now, assume $M$ has more rows than columns.
    %Suppose that we increase $n_t$ by 1 for some $t$. In this way we add precisely r columns (indexed by $(t,m,n_t+1)$ for all $m\in[r]$) to $M$. Note that both the number of columns and the number of rows of $M$ are divisible by $r$. Therefore, by increasing $n_t$ by some number we can increase the number of columns to match the number of rows, thus getting a square matrix. Then by \autoref{supermegatheorem} it has full row rank and, since it is a square matrix, it has full rank. Thus, all columns of the new matrix are linearly independent, hence all columns were linearly independent in the original matrix $M$. 
\end{proof}

\section{Further research directions} 

In this section we highlight several possible direction for further research.

\subsection{Generic rank}
\label{genrank}
While the asymptotic growth of a generic rank and its values for tensors of certain shapes (e.g., in $K^{n,n,n}$) are known, exact values are still unknown in general \cite{FRIEDLAND2012478}.
In the spirit of this paper, one can rephrase the question of finding exact values of a generic rank to a question about the generic rank of a matrix by taking the differential of the map $\psi_r:$ 
\begin{align*}
    \psi_r: \left(K^{n_1}\times K^{n_2}\times K^{n_3}\right)^r &\to K^{n_1,n_2,n_3}, \\
    (a^{1,1},a^{2,1},a^{3,1},\; a^{1,2},a^{2,2},a^{3,2},\; \dots, \; a^{1,r},a^{2,r},a^{3,r}) &\mapsto \sum_{i=1}^r \, a^{1,i}\otimes a^{2,i}\otimes a^{3,i}.
\end{align*}
Note that the structure of the matrix of the differential is similar to that of $M$.

It is conjectured (see \cite{Strassen83} and  \cite[Conjecture 5.1]{FRIEDLAND2012478})
that when $3 \leq n_1 \leq n_2 \leq n_3 \leq (n_1 - 1)(n_2 - 1)$, then apart from the case $n_1=3, n_2=n_3=2k+1$, the generic rank is equal to 
$$R(n_1,n_2,n_3) = \left\lceil\frac{n_1n_2n_3}{n_1+n_2+n_3 - 2}\right\rceil,$$
which is the expected value.

One can ask whether a combinatorial method similar to the one developed in this paper, possibly with an adapted notion of a unique monomial, could be used to prove that the matrix has expected rank.

\subsection{Equations of the varieties $\overline{\mathcal{C}_r}$}
Another interesting problem to consider concerns the description of the ideal $I(\mathcal{C}_r)$. Note that it is a difficult problem as, for instance, $\overline{\mathcal{C}_r} \subset K^{r,r,r} = \{T \in K^{r,r,r} \: | R(T) \leq r\}$ which is the $r$-th secant variety of the Segre embedding. Describing the equations of secants of Segre varieties  is a major open problem as these varieties are loci of tensors of border rank at most $r$ (see \cite{Landsberg2006GeneralizationsOS} and \cite[Chapter 7]{JML} for more information). 

We note that in general there are still some similarities with the case of secants of Segre varieties since $\overline{\mathcal{C}_r} \subset K^{n,n,n}$ is invariant under the action of $GL_n \times GL_n \times GL_n$ so we can study $I(\overline{\mathcal{C}_r})$ with the tools from representation theory.
%The first different from secants would be $\mathcal{C}_4$ in $K^{5,5,5}$, right?

\subsection{(Generic) border subrank}

For a tensor $T \in K^{n_1,\ldots, n_k}$ we define its \textit{border subrank} $\underline{Q}(T)$ as the maximal non-negative integer $s \leq n$ such that $$I_s \in \overline{\prod_{i=1}^k GL_{n_i}(K)(T)}$$ where $I_s := I_1^{\oplus s} = \sum_{j=1}^s \bigotimes_{i=1}^{k}e_j^{(i)}$ and $\{e_{j}^{(i)}\}_{j=1}^{n_i}$ is a basis of the corresponding $K^{n_i}$ for each $1 \leq i \leq k$. Note that this definition is independent of the choice of the bases and we always have $Q(T) \leq \underline{Q}(T)$. Then we can define sets $$\underline{\mathcal{C}}_r^{(n_1,\ldots,n_k)} := \{T \in K^{n_1,\ldots,n_k} \: | \: \underline{Q}(T) \geq r\}.$$ It follows from the definition that $\mathcal{C}_r^{(n_1,\ldots,n_k)} \subset \underline{\mathcal{C}}_r^{(n_1,\ldots,n_k)}$. Moreover, one can prove a similar upper bound as in Theorem \ref{dimension} for the dimensions of the sets $\underline{\mathcal{C}}_r^{(n_1,\ldots,n_k)}$ (see \cite[Theorem 3]{biaggi2024bordersubrankgeneralisedhilbertmumford}) which, in fact, yields the same asymptotic growth for the generic border subrank \cite[Theorem 2]{biaggi2024bordersubrankgeneralisedhilbertmumford}: $$\underline{Q}(n,n,\ldots,n) = \Theta(n^{1/(k-1)}).$$ In \cite[Theorem 4]{biaggi2024bordersubrankgeneralisedhilbertmumford} it is shown that the generic border subrank in $K^{n,n,n}$ is at least $\lfloor \sqrt{4n} \rfloor - 3$, which is bigger then the generic subrank for large enough $n$. An interesting research direction is to find the exact values (or at least give bounds similar to \cite[Theorem 1.2 and Theorem 1.3]{10.4230/LIPIcs.CCC.2022.9}) of the generic border subrank. 

On the other hand, determining dimensions of $\underline{\mathcal{C}}_r^{(n_1,\ldots,n_k)}$ seems to be a challenging problem as the only lower bound we are aware of is \cite[Theorem 1.4]{maxbrdrsubrank}: $$\dim(\underline{\mathcal{C}}_r^{(r,r,r)}) \geq \frac{2r^3 + 3r^2 - 2r - 3}{3}.$$

\bibliographystyle{alpha}
\bibliography{bibl}

\textsc{Paweł Pielasa, Faculty of Mathematics, Informatics and Mechanics, University of Warsaw}\\
\textit{email address:} p.pielasa@student.uw.edu.pl

\textsc{Matouš Šafránek, Faculty of Mathematics and Physics, Charles University, Prague}\\
\textit{email address:} matous.safranek@gmail.com.

\textsc{Anatoli Shatsila, Institute of Mathematics, Jagiellonian University in Krak\'ow, Poland}\\
\textit{email address:} anatoli.shatsila@doctoral.uj.edu.pl.

\end{document}